\documentclass[11pt]{article}
\usepackage[top=0.75in, bottom=1in, left=3cm, right=3cm]{geometry}
\usepackage{cite}

\usepackage{amssymb,amsfonts,amsmath,amsthm}
\allowdisplaybreaks[4]
\usepackage{esint}
\usepackage{braket}
\usepackage{extarrows}
\usepackage{tensor}
\usepackage{bm}
\usepackage[shortcuts]{extdash}
\usepackage{mathrsfs}
\everymath{\displaystyle}

\numberwithin{equation}{section}
\usepackage{titlesec}
\usepackage{titletoc}
\usepackage{appendix}
\usepackage[subfigure]{tocloft}

\usepackage{caption}
\usepackage[dvips]{graphicx}
\usepackage{graphics}
\usepackage{overpic}
\usepackage{subfigure}
\usepackage{color}
\usepackage{framed}
\usepackage{tikz}
\usetikzlibrary{cd}
\tikzset{every picture/.style={line width=0.75pt}} 
\definecolor{shadecolor}{rgb}{0.92,0.92,0.92}
\usepackage{listings}
\usepackage{float}
\usepackage{tabularx}
\usepackage{pgfplots}
\pgfplotsset{compat=1.15}
\usepackage[colorlinks=true,linkcolor=black,citecolor=black,CJKbookmarks=true]{hyperref}
\usepackage{footmisc}
\usepackage{marvosym}
\usepackage{verbatim}

\usepackage{thmtools}
\declaretheoremstyle[
bodyfont=\itshape
]{mystyle}
\declaretheorem[style=mystyle,name=Theorem]{theorem}
\declaretheorem[style=mystyle,name=Lemma,sibling=theorem]{lemma}
\declaretheorem[style=mystyle,name=Corollary,sibling=theorem]{corollary}
\declaretheorem[style=mystyle,name=Proposition,sibling=theorem]{proposition}

\declaretheoremstyle[
]{mystyle1}

\declaretheorem[style=mystyle,name=Definition,sibling=theorem]{definition}
\declaretheorem[style=mystyle1,name=Remark,sibling=theorem]{remark}

\makeatletter
\renewcommand{\@fnsymbol}[1]{\arabic{footnote}}
\makeatother

\allowdisplaybreaks[4]
\usepackage{cite}		

\begin{document}
	\captionsetup[figure]{labelfont={bf},name={Fig.},labelsep=period}
	\captionsetup[table]{labelfont={bf},name={Fig.},labelsep=period}
	\boldmath
	\title{\centering\textbf{An asymptotic slope of the $\alpha$-K-energy for the K\"{a}hler-Yang-Mills equations}}
	\author{By \text{Akito Futaki\footnote{Yau Mathematical Sciences Center, Tsinghua University, Beijing 100084, China, futaki@tsinghua.edu.cn} } and \text{Jianwei Shi\footnote{Department of Mathematical Sciences, Tsinghua University, Beijing 100084, China, shijw23@mails.tsinghua.edu.cn}}}
\date{\today}
\maketitle
\unboldmath
\begin{abstract}
	For holomorphic vector bundles over compact Kähler manifolds, we establish a formula for the asymptotic slope of the $\alpha$-K-energy associated with the K\"{a}hler-Yang-Mills equations.
\end{abstract}
\section{Introduction}
\label{sec:1}

In this paper we study a system of partial differential equations that combines the Hermitian-Einstein equation and the cscK equation, namely the K\"{a}hler-Yang-Mills equations, which was introduced in \cite{garcia2011coupled} and \cite{alvarez2013coupled} firstly. More precisely, let $G$ be a compact Lie group, and $G^c$ its complexification, we consider a $G^c$-principal bundle $ P $ (with $G^c$-action on the right) over a compact K\"{a}hler manifold $ X $, endowed with a  $G^c$-equivariant integrable holomorphic structure $I$  such that the projection $\pi:P\to X$ is holomorphic. The Kähler-Yang-Mills equations form a coupled system that relates two geometric quantities on \(X\) and $P$: the scalar curvature \(S_g\) of a Kähler metric \(g\), and the curvature \(F_H\) of the Chern connection (in the sense of \cite{singer1959geometric}) associated with a \(G\)-reduction \(H\) on \((P,I)\), which can be described as follows.

Let $\mathfrak{g}$ and $\mathfrak{g}^c$ be Lie algebras of $G$ and $G^c$ respectively. We fix an inner product $(\cdot,\cdot)$ on $\mathfrak{g}$ which is invariant under the adjoint action, and extend it to a $G^c$-adjoint-invariant symmetric bilinear map $(\cdot,\cdot):\mathfrak{g}^c\otimes\mathfrak{g}^c\to \mathbb{C}$. This induces a $\mathbb{C}$-bilinear product on $\operatorname{ad}P$, the adjoint bundle of $P$ (c.f. \cite[volume 1, III.1]{kobayashinomizu1969foundations}). We then extend this product to a $\mathbb{C}$-bilinear map $\wedge$ on $\Omega^*(\operatorname{ad}P)$ by\begin{align*}
	\wedge: \Omega^p(\operatorname{ad}P)\times \Omega^q(\operatorname{ad}P)&\to \Omega^{p+q},\\
	(a,b)\mapsto a\wedge b,
\end{align*}where $\Omega^p(\operatorname{ad}P)$ is the space of $p$-forms with coefficient in $\operatorname{ad}P$. 

We say that a K\"{a}hler form $\omega$ on $X$ and a $G$-reduction $H$ of $P$ satisfy \textit{the K\"{a}hler-Yang-Mills equations} with coupling constant $\alpha_0,\alpha_1\in\mathbb{R}$ if
\begin{align}\label{KYM}
	\begin{cases}
		\Lambda_\omega F_H=z,\\
		\alpha_0S_\omega+\alpha_1\Lambda_\omega^2(F_H\wedge F_H)=c,
	\end{cases}
\end{align}where $\Lambda_\omega$ is the contraction with respect to the K\"{a}hler form $\omega$; $z\in \mathfrak{g}$ is invariant under the adjoint action and thus can be regarded as a constant section of $\operatorname{ad}P$; $S_\omega$ is the scalar curvature of the K\"{a}hler metric respect to $\omega$; and both $z$ and $c$ only depend on the topology of $P$ and  the cohomology class of $\omega$. Note that $c$ is a real number since in fact the curvature $F_H$ takes values in $\mathfrak{g}$, which is a real Lie algebra.

In the classical theory of the Hermitian-Einstein equation and in the theory of the cscK equation, it is well known that the existence of solutions for both equations is closely linked to a corresponding kind of algebraic stability condition. In the theory of Hermitian-Einstein metrics, the Donaldson-Uhlenbeck-Yau theorem (also known as the Kobayashi-Hitchin correspondence, c.f. \cite{kobayashi2014differential, Lubke1983, donaldson1985anti, uhlenbeck1986existence}) states that, for a holomorphic vector bundle $E$ over a compact Kähler manifold, $E$ admits a Hermitian-Einstein metric if and only if it is slope polystable. In the context of the cscK equation, the Yau-Tian-Donaldson (YTD) conjecture (c.f. \cite{donaldson2002scalar, tian1997kahler, yau}) predicts that a polarized algebraic manifold $(X,L)$ is K-polystable if and only if the polarization class $c_1(L)$ admits a Kähler metric of constant scalar curvature, which was proved in the Fano case by \cite{chen2015kahler,chen2015kahler2, chen2015kahler3, tian1997kahler}.

The aim of this paper is to provide	a stability condition from the variational point of view of a functional $\mathcal{M}_I$ called \textit{the $\alpha$-K-energy} constructed in \cite{garcia2011coupled,alvarez2013coupled}. Our first result is the formulation of the $\mathcal{M}_I$. Let $(X,\omega)$ be a compact K\"{a}hler manifold, and $\pi:E\to X$ an irreducible holomoprhic vector bundle over $X$. Consider the space of Kähler potentials, denoted by $\mathcal{H}_{\omega} = \{\varphi\in C^{\infty}(X):\omega_{\varphi} = \omega + dd^c\varphi > 0\},$ and the space of Hermitian metrics on $E$, denoted by $\operatorname{Herm}^+ (E)$. Then, for any $b_i = ( \omega_i, h_i)\in\mathcal{H}_{\omega} \times \operatorname{Herm}^ + (E),i = 0, 1$, we have
\begin{equation}
	\begin{aligned}\label{mytheorem1}
		&\mathcal{M}_I(b_1,b_0)=2\alpha_1\operatorname{M}^{\operatorname{Don}}_{\omega_1}(h_1,h_0)+\alpha_0\operatorname{M}^{\operatorname{cscK}}(\omega_1,\omega_0)
		\\ &+ \frac{\alpha_1\overline{C}_1-4\alpha_1\overline{C}_2}{(n+1)!}\braket{\varphi_1,...,\varphi_1}_{(\omega_0,...,\omega_0)}-4\pi\alpha_1\frac{1}{(n-1)!}}\braket{\varphi_1,\dots,\varphi_1}_{(\operatorname{ch}_2(E,h_0))(\omega_0,\dots,\omega_0)\\&+ 8\pi\alpha_1\frac{\mu(E)}{n!\operatorname{vol}_\omega}\braket{0,\varphi_1,\varphi_1,...,\varphi_1}_{(c_1(E,h_0),\omega_0,...,\omega_0)},
	\end{aligned}
\end{equation}where $\operatorname{M}^{\operatorname{Don}}_{\omega_1}$ is the Donaldson functional with respect to the K\"{a}hler form $\omega_1$, $\operatorname{M}^{\operatorname{cscK}}$ is the Mabuchi K-energy,  $\overline{C}_1,\overline{C}_2$ are topological invariants, and $\varphi_1$ is a smooth function such that $\omega_1=\omega_0+dd^c\varphi_1$. The symbols $\braket{\cdot,\dots,\cdot}_{(\cdot,\dots,\cdot)}$ and $\braket{\cdot,\dots,\cdot}_{(\cdot)(\cdot,\dots,\cdot)}$ represent the multivariate energy functional and its modified version, respectively defined by (\ref{delignepairing}) and (\ref{modified}).

This formula inspires us to analyze the $\alpha$-K-energy by splitting it into two parts. We will refer to $\alpha_0\operatorname{M}^{\operatorname{Don}}_{\omega_1}(h_1,h_0)$ as the `HE' component, denoted by $\mathcal{M}^{\operatorname{HE}}(b_1,b_0)$, and the rest as the `cscK' component, denoted by $\mathcal{M}^{\operatorname{cscK}}(b_1,b_0)$. Note that if $h_1=h_0$, then $\mathcal{M}^{\operatorname{cscK}}(b_1,b_0)=\mathcal{M}_I(b_1,b_0)$.

We first study the asymptotic behaviour of the $\alpha$-K-energy when the Hermitian metric $h$ is fixed. Along smooth $\mathcal{H}_\omega$-subgeodesic ray compatible (in sense of \cite{sjostrom2018k}) with ample algebraic test configuration for a polarized complex manifold $(X,L)$, the asymptotic slope of $\mathcal{M}_I$ can be represented as the intersection number of cohomology classes plus the non-Archimedean Mabuchi functional. To be more precise, we establish the following result:\begin{theorem}\label{my}
	Let $(X,L)$ be a polarized complex manifold, $(\mathcal{X},\mathcal{L})$ be an ample test conﬁguration for $(X,L)$ which is smooth and dominates $X\times \mathbb{P}^1$ with birational morphism $\mu$, and $\pi:E\to X$ an irreducible holomorphic vector bundle over $X$. Then for any Hermitian metric $h$ of $E$, and any smooth strictly positive metric $\Phi$ on $\mathcal{L}$ near the central ﬁber, we have
	\begin{equation}\label{McscK}
		\begin{aligned}
			\lim_{t\to\infty}\frac{\mathcal{M}_I(b_t,b_0)}{t}&=\alpha_0\operatorname{M}^{\operatorname{NA}}(\mathcal{X}, \mathcal{L})-4\pi\alpha_1\frac{1}{(n-1)!}(\mu^*\operatorname{pr}_1^*\operatorname{ch}_2(E)\cdot \mathcal{L}\cdot\dots\cdot \mathcal{L}) \\&\quad\quad+8\pi^2\alpha_1\frac{\mu(E)}{n!\operatorname{vol}_\Omega}(\mu^*\operatorname{pr}_1^*c_1(E)\cdot \mathcal{L}\cdot\dots\cdot \mathcal{L})\\& \quad\quad+\frac{\alpha_1\overline{C}_1-4\alpha_1\overline{C}_2}{(n+1)!}(\mathcal{L}\cdot\dots\cdot \mathcal{L}),
		\end{aligned}
	\end{equation}where $b_t = (\omega_t, h)$ is a ray of pairs consisting of K\"{a}hler form in $c_1(L)$ and Hermitian metric on $E$, with $\omega_t = dd^c\phi_t$ for the ray of smooth positive metrics $\phi_t$ on $L$ corresponding to $\Phi$. Here, $\operatorname{M}^{\operatorname{NA}}(\mathcal{X},\mathcal{L})$ denotes the non-Archimedean Mabuchi functional of the test configuration $(\mathcal{X},\mathcal{L})$ (c.f. \cite{boucksom2017uniform,boucksom2019}).
\end{theorem}
For the `HE' part, as analogous to the test configuration, we consider filtration of subbundles or subsheaves. As shown in \cite[Theorem 3.3]{jonsson2022geodesicraysdonaldsonuhlenbeckyautheorem}, any filtration of subbundles induces a type of smooth $\operatorname{Herm}^+(E)$-geodesic ray. Along any such ray, one can explicitly compute the curvature and then Donaldson's functional.

The only difference in our case is that, in the `HE' part $\mathcal{M}^{\operatorname{HE}}(b_t,b_0)=\alpha_0\operatorname{M}^{\operatorname{Don}}_{\omega_t}(h_t,h_0)$, the base K\"{a}hler space $(X,\omega_t)$ changes with $h_t$. Thanks to \cite[Theorem 3.3]{jonsson2022geodesicraysdonaldsonuhlenbeckyautheorem}'s calculations, we can overcome this difficulty and get:
\begin{theorem}\label{mytheorem2}
	Suppose that $(\omega_t)_{t\geq 0}$ is a smooth ray in the space $\mathcal{K}$ of K\"{a}hler forms in a given cohomological class $\alpha$. Let $h_0$ be a hermitian metric on $E$ and
	\[
	0 =: E_{m+1} \subset E_m \subset \dots \subset E_1 := E
	\]
	a filtration of $E$ by holomorphic subbundles. Let $F_i := E_i/E_{i+1}$, $1 \leqslant i \leqslant m$. Then there exists $w \in \operatorname{Herm}(E,h_0)$, such that the geodesic ray of hermitian metrics $h_t := e^{tw}h_0$, $t \geqslant 0$, satisfies:
	\begin{equation}\label{MDon}
		\lim_{t \to \infty} \frac{\operatorname{M}_{\omega_t}^{\operatorname{Don}}(h_t, h_0)}{t} = \sum_{k=1}^m 2\pi(m - k + 1)\operatorname{rank}(F_k)(\mu_{F_k} - \mu_E).
	\end{equation}
	Combining this result with Theorem \ref{my} will give us the complete slope of $\mathcal{M}_I$ along the $\mathcal{H}_\omega\times \operatorname{Herm}^+(E)$-ray $b_t=(\omega_t,h_t)$, i.e.\begin{equation}\label{finial}
		\begin{aligned}
			\lim_{t \to \infty} \frac{\mathcal{M}_I(b_t, b_0)}{t} =&\ 4\pi\alpha_1\sum_{k=1}^m (m - k + 1)\operatorname{rank}(F_k)(\mu_{F_k} - \mu_E)+\alpha_0\operatorname{M}^{\operatorname{NA}}(\mathcal{X}, \mathcal{L})
			\\&-4\pi\alpha_1\frac{1}{(n-1)!}(\mu^*\operatorname{pr}_1^*\operatorname{ch}_2(E)\cdot \mathcal{L}\cdot\dots\cdot \mathcal{L}) \\&+8\pi\alpha_1\frac{\mu(E)}{n!\operatorname{vol}_\Omega}(\mu^*\operatorname{pr}_1^*c_1(E)\cdot \mathcal{L}\cdot\dots\cdot \mathcal{L})\\&+\frac{\alpha_1\overline{C}_1-4\alpha_1\overline{C}_2}{(n+1)!}(\mathcal{L}\cdot\dots\cdot \mathcal{L}).
		\end{aligned}
	\end{equation}
\end{theorem}
Finially note that the first equation in $(\ref{KYM})$ is just the Hermitian-Einstein equation. Therefore, if a solution to $(\ref{KYM})$ exists, then $E$ should be slope-stable. This means that the slope $(\ref{MDon})$ of the `HE' part in $\mathcal{M}_I$ should be positive. However, the slope (\ref{McscK}) of the `csck' part may not be non-negative when a solution exists. It is unreasonable to expect that $\mathcal{M}^{\operatorname{cscK}}$ is convex along the weak geodesic in $\mathcal{H}_\omega$, because it involves $E$'s topological invariant. Actually, a weak geodesic for (\ref{KYM}) should be a weak solution of \cite[(3.128)]{alvarez2013coupled}. This suggests that considering only $\mathcal{H}_\omega$-ray might be insufficient, and it is necessary to provide the complete slope (\ref{finial}) of $\mathcal{M}_I$ along $\mathcal{H}_\omega\times \operatorname{Herm}^+(E)$-ray.
\paragraph{Organisation of the paper.} We begin by recalling some backgrounds on the K\"{a}hler-Yang-Mills equations \eqref{KYM} and by reviewing the definition of the $\alpha$-K-energy in section \ref{sec:2}. In section \ref{sec:3}, we introduce the key tool used in our paper, namely the \textit{(modified) multivariate energy functional}, which is introduced in \cite{sjostrom2018k} and extended by \cite{Dervan2023}. At the end of section \ref{sec:3}, we discuss the asymptotic slope of the modified multivariate energy functional. To introduce this result, we also recall two notions from \cite{sjostrom2018k}: \textit{cohomological test configurations} (which extend the concept of algebraic test configurations) and \textit{$\mathcal{C}^\infty$-compatibility}. In section \ref{sec:4} we fix a particular choice of path in order to compute the value of the $\alpha$-K-energy, by which we decompose the $\alpha$-K-energy into four parts $Q_1',Q_2',\operatorname{M}',\operatorname{M}^{\operatorname{cscK}}.$ This leads to expressions for $Q_1'$ and $Q_2'$ written in terms of the modified multivariate energy functional, which are analogous to the Chen-Tian formula \cite{chen2000lower}. We also relate $\operatorname{M}'$ to the Donaldson functional by presenting four equivalent characterizations of the space of Hermitian metrics on an irreducible holomorphic vector bundle. In section \ref{subsec:7} we combine the results obtained in section \ref{sec:3} and \ref{sec:4} to deduce our formula (\ref{mytheorem1}). We then apply this approach to the algebraic case considered in \cite{boucksom2019} to establish Theorem \ref{my}, alongside investigating a cohomological case as Theorem \ref{mycor}. Finally, in section \ref{subsec:8}, with the help of \cite{jonsson2022geodesicraysdonaldsonuhlenbeckyautheorem}, we evaluate the slope of the `HE' part, and bring together results in preceding sections to yield Theorem \ref{mytheorem2}.
\paragraph{Acknowledgments}We would like to thank Mattias Jonsson for directing our notice to \cite{jonsson2022geodesicraysdonaldsonuhlenbeckyautheorem}. We would also like to express our gratitude to Mario Garc\'{i}a-Fern\'{a}ndez and Oscar Garc\'{i}a-Prada for their support and interest in this project.

\boldmath
\section{Coupled K\"{a}hler-Yang-Mills equations and $\alpha$-K-energy}
\unboldmath
\label{sec:2}
In this section let us recall some basic facts of K\"{a}hler-Yang-Mills equations from the foundational work \cite{alvarez2013coupled}. Let $(X,\check{I})$ be a compact complex manifold of dimension $n$ for which $\Omega\in H^2(X,\mathbb{R})$ is a K\"{a}hler class, $G$ a real compact Lie group with complexification $G^c$, and denote their Lie algebras by $\mathfrak{g}$ and $\mathfrak{g}^c$ respectively. Let $P$ be a principle $G^c$-bundle ($G^c$-action on the right) with integrable $G^c$-equivariant almost complex structures $I$ such that the projection $\pi:P\to X$ is holomorphic, which means that the complex structures $I$ can descend to $\check{I}$ on $X$ in such a way that $d\pi\circ I=\check{I}\circ d\pi.$

Let $\mathcal{R}$ be the space of all $G$-reductions of $P$. It is identified with the space $H^0(P/G)$
of sections of $P/G \to X$, via the bijection $H\in H^0(P/G)\mapsto P_H=p^{-1}(H(X))$, where $p:P\to P/G$ is the natural projection. Note that $(\cdot,\cdot)$ in the introduction can also induce an $\mathbb{R}$-bilinear product on $\Omega^*(\operatorname{ad}P_H)$, which is the restriction of $\wedge$ on $\Omega^*(\operatorname{ad}P)$ if we embed $\Omega^*(\operatorname{ad}P_H)$ to $\Omega^*(\operatorname{ad}P)$ by $P_H\subset P$. Actually $P/G$ is a $\mathfrak{g}$-affine bundle on $X$ that is isomorphic to the vector bundle $\operatorname{ad}P_H$ when we fix $H$ as zero section (see \cite[(3.120)]{alvarez2013coupled}), and thus $\mathcal{R}$ is an affine space and contractible.

Giving a $G$-reduction $H\in\mathcal{R}$ is equivalent to specifying a `Hermitian metric' on $P$. Such a reduction, together with holomorphic structure $I$, can induce a Chern connection in the sense of \cite{singer1959geometric}. We denote this connection by $A_H$. It can be reduced to the subbundle $P_H\subset P$ and its curvature, written as $F_H\in \Omega^2(\operatorname{ad}P)$ for the latter, is a  $(1,1)$-form. Actually such relation between reductions, holomorphic structures and connections, gives a biholomorphic map from $Z_b$ to $\mathcal{P}_b$, once given $b=(\omega,H)$, a pair of a symplectic form on $X$ and a $G$-reduction. Here $Z_b$ is the space of all $G^c$-equivariant integrable holomorphic structures  on $P$ whose descents  on $X$ is are compatible with $\omega$; and $\mathcal{P}_b$ is the space of all pairs $(J,A)$ such that $J$ is an integrable holomorphic structure on $X$ compatible with $\omega$ and $A$ is a connection on the $G$-reduction $P_H$ of $P$ (see \cite[Lemma 3.1]{alvarez2013coupled}). Both $Z_b$ and $\mathcal{P}_b$ have holomorphic structures in the sense of \cite[(3.62)]{alvarez2013coupled} and \cite[(2.45)]{alvarez2013coupled}.

The space $\mathcal{K}_{{\check{I}}}$ of K\"{a}hler forms on $(X,\check{I})$ representing the cohomology class $\Omega$ is also contractible, modelled on $\mathcal{H}_\omega/\mathbb{R}$, where  $\mathcal{H}_\omega:=\{\varphi\in\mathcal{C}^\infty(X):\omega_\varphi:=\omega+dd^c\varphi>0\}$, once we fix some $\omega\in \mathcal{K}_{{\check{I}}}$ as a base point. Recall that $\mathcal{H}_\omega$ is convex in $\mathcal{C}^\infty(X)$, and consists of all strictly $\omega$-psh functions. As is well-known, the space $\mathcal{H}_\omega$ carries a natural structure of an infinite-dimensional symmetric space, with the Levi-Civita connection determined by the $L^2$-metric on K\"{a}hler potentials. By \cite[Theorem 3.6]{alvarez2013coupled}, the product space $B_I:=\mathcal{K}_{\check{I}}\times \mathcal{R}$ also carries an infinite-dimensional symmetric space structure, induced by a connection $\nabla$ on a principal bundle lying over it. This construction is the analogue of \cite[Section 4, Proposition 2]{donaldson1999symmetric}.

Fix a coupling constants $\alpha=(\alpha_0,\alpha_1),\alpha_i>0,i=1,2$, we say a pair $(\omega,H)\in B_I$ satisﬁes the K\"{a}hler-Yang-Mills equations if \eqref{KYM} holds for $(\omega,H)$. 
For the existence of such pair of solutions, \cite{alvarez2013coupled} introduced a functional analogous to the Mabuchi K-energy (see Proposition 3.10 and (3.127) there), providing a variational framework 
for the equations \eqref{KYM}:\begin{theorem}[\text{\citenum{alvarez2013coupled}, Proposition 3.10}]\label{alphaKenergy}Let the functional $\mathcal{M}_{I}(\cdot,\cdot):B_I\times B_I\to \mathbb{R}$ defined by
	\begin{align*}
		\mathcal{M}_{I}\left((\omega_1,H_1),(\omega_0,H_0)\right)=&\ 4\alpha_1\int_{0}^{1}\int_X\sqrt{-1}\dot{H}_s\wedge (\Lambda_{\omega_s}F_{H_s}-z)\frac{\omega_s^n}{n!}\wedge ds\\&+\frac{1}{2\pi}\int_{0}^{1}\int_X\dot{\varphi}_s(S_\alpha(\omega_s,H_s)-\overline{S}_\alpha)\frac{\omega_s^n}{n!}\wedge ds,
	\end{align*}
	where $(\omega_s,H_s),s\in [0,1]$ is a path from $(\omega_0,H_0)$ to $(\omega_1,H_1)$, $\varphi_s$ is a ray of smooth function such that $\omega_s=\omega_0+dd^c\varphi_s$, the coupled curvature $S_\alpha(\omega_s,H_s)$ is given by\begin{align*}
		S_\alpha(\omega,H):=-\alpha_0S_{\omega}-\alpha_1\Lambda_{\omega}^2(F_H\wedge F_H)+4\alpha_1\Lambda_\omega F_H\wedge z.
	\end{align*}and $\overline{S}_\alpha=\frac{\int_XS_\alpha(\omega,H)\omega^n}{\int_X\omega^n}$ is a constant independent of the choice of $\omega$ and $H$. Then the following holds:\\
	(i) $\mathcal{M}_I\left((\omega_1,H_1),(\omega_0,H_0)\right)$ is independent of the choice of the path $(\omega_s,H_s),s\in [0,1]$.\\
	(ii) Suppose that $(B_I,\nabla)$ is geodesically convex. If \eqref{KYM} has a solution, then for any fixed $(\omega_0,H_0)$, the functional $\mathcal{M}_I(\cdot,(\omega_0,H_0))$ is bounded from below.
\end{theorem}
\begin{remark}
	The original definition of $\mathcal{M}_I$ in \cite{alvarez2013coupled, garcia2011coupled} does not contain $\overline{S}_\alpha$, because the curve $(\omega_s,H_s)$ taken there satisfies $\int_X\dot{\varphi}_t\omega^n_t=0$ (i.e. the curve in $B_I=\mathcal{K}_{\check{I}}\times \mathcal{R}$). Here, we extend it to $\mathcal{H}_{\omega_0}\times \mathcal{R}$ and we can prove that $\mathcal{M}_I$ is invariant if we replace $\varphi_t$ by $\varphi_t+C(t)$.
	
	The second term of $\mathcal{M}_I$ also has an additional coefficient $\frac{1}{2\pi}$. This is because the $d^c$ we use is $\frac{\sqrt{-1}}{2\pi}(\bar{\partial}-\partial)$ (same as \cite{boucksom2019}), while in  \cite{alvarez2013coupled}, it is equal to $\sqrt{-1}(\bar{\partial}-\partial)$ (see the proof of \cite[Proposition 3.16]{alvarez2013coupled}). Also, for the Mabuchi energy, the non-Archimedean Mabuchi functional, and the Donaldson-Futaki invariant, we use the definition in \cite{boucksom2019} multiplied by $\frac{\operatorname{vol}_\omega}{2\pi}$.
\end{remark}
The functional $\mathcal{M}_I$ is called the $\alpha$-K-energy, which plays a central role in our paper. We want to provide a new formulation of $\mathcal{M}_I$ using a method similar to \cite{chen2000lower}, and then analyze its slope with the help of \cite{sjostrom2018k} and \cite{jonsson2022geodesicraysdonaldsonuhlenbeckyautheorem}. For later use we write\begin{align*}
	b_s&=(\omega_s,H_s),\\Q_1(b_1,b_0)&:=\frac{1}{2\pi}\int_{0}^{1}\int_X\dot{\varphi}_s\left(\Lambda_{\omega_s}^2(F_{H_s}\wedge F_{H_s})-2\pi\overline{C}_1\right)\frac{\omega_s^n}{n!}\wedge ds,\\Q_2(b_1,b_0)&:=\frac{1}{2\pi}\int_{0}^{1}\int_X\dot{\varphi}_s\left(\Lambda_{\omega_s}F_{H_s}\wedge z-2\pi\overline{C}_2\right)\frac{\omega_s^n}{n!}\wedge ds,\\
	\operatorname{M}(b_1,b_0)&:=\int_{0}^{1}\int_X\sqrt{-1}\dot{H}_s\wedge (\Lambda_{\omega_s}F_{H_s}-z)\frac{\omega_s^n}{n!}\wedge ds,
\end{align*}where $\overline{C}_1=\frac{\int_X\Lambda_{\omega}^2(F_{H}\wedge F_{H})\frac{\omega^{n}}{n!}}{2\pi\int_X\frac{\omega^n}{n!}},\overline{C}_2=\frac{\int_X\Lambda_{\omega}(F_{H}\wedge z)\frac{\omega^{n}}{n!}}{2\pi\int_X\frac{\omega^n}{n!}}$ are constants that are independent of the choice of $\omega$ and $H$. Then one has\begin{align}
	\mathcal{M}_I(b_1,b_0)=4\alpha_1\operatorname{M}(b_1,b_0)+\alpha_0\operatorname{M}^{\operatorname{cscK}}(\omega_1,\omega_0)-\alpha_1Q_1(b_1,b_0)+4\alpha_1Q_2(b_1,b_0),
\end{align}where $\operatorname{M}^{\operatorname{cscK}}$ is the Mabuchi K-energy. By the classical theory of cscK equation, $\operatorname{M}^{\operatorname{cscK}}$ is independent of the choice of path (c.f. \cite[Theorem 2.4]{mabuchi1986k}), thus
\begin{corollary}\label{independent}
	$4\operatorname{M}+4Q_2-Q_1$ is independent of the choice of path.
\end{corollary}
\section{Multivariate energy functionals and their asymptotic slope}
\label{sec:3}
\subsection{Multivariate energy functionals and modified multivariate energy functionals}
\label{subsec:2}
We now introduce the key tool employed in our paper, namely the \textit{multivariate energy functional} or {\textit{analytic Deligne pairing},} introduced in \cite{sjostrom2018k} by Sj\"{o}str\"{o}m Dyrefelt. In the theory of the constant scalar curvature K\"{a}hler (cscK) equation, it provides a remarkably refined means of relating the asymptotic slope of Mabuchi functional $\operatorname{M}^{\operatorname{cscK}}$ with $\operatorname{M}^{\operatorname{NA}}$, the non-Archimedean Mabuchi functional (a modiﬁcation of the Donaldson-Futaki invariant, introduced by \cite{boucksom2017uniform,boucksom2019}).

In this section, we assume $X$ is a compact K\"{a}hler manifold of dimension $n$, equipped with a K\"{a}hler form $\omega$.
\begin{definition}[\text{\citenum{sjostrom2018k}, Definition 2.1}]
	Let $\theta_0,\dots,\theta_n$ be closed $(1,1)$-forms on $X$.  
	Define a multivariate energy functional $\braket{\cdot,\dots,\cdot}_{(\theta_0,\ldots,\theta_n)}$  
	on the space 
	$\mathcal{C}^\infty(X)
	\times\dots\times 
	\mathcal{C}^\infty(X)$  
	($n+1$ times) by
	\begin{align}\label{delignepairing}
		\langle \varphi_0,\dots,\varphi_n\rangle_{(\theta_0,\dots,\theta_n)}
		&:= 
		\int_X \varphi_0\,(\theta_1 + dd^c\varphi_1)\wedge\dots\wedge(\theta_n + dd^c\varphi_n)\notag \\&+\int_X \varphi_1\,\theta_0\wedge(\theta_2 + dd^c\varphi_2)\wedge\dots\wedge(\theta_n + dd^c\varphi_n)\notag\\& +\int_X \varphi_2\,\theta_0\wedge\theta_1\wedge(\theta_3 + dd^c\varphi_3)\wedge\dots\wedge(\theta_n + dd^c\varphi_n)
		\notag\\ &+ \dots
		+ \int_X \varphi_n\,\theta_0\wedge\dots\wedge\theta_{n-1}.
	\end{align}
\end{definition}
\noindent This functional enjoys the following properties: 
\begin{proposition}[\text{\citenum{sjostrom2018k}, Proposition 2.3 \& 2.8}]\mbox{}\label{duichen}\\
	(i) $\braket{\cdot,\dots,\cdot}_{(\theta_0,\ldots,\theta_n)}$ is symmetric, i.e. for any $\sigma\in S_{n+1}$, $$\braket{\varphi_{\sigma(0)},\dots,\varphi_{\sigma(n)}}_{(\theta_{\sigma(0)},\dots,\theta_{\sigma(n)})}=\braket{\varphi_0,\dots,\varphi_n}_{(\theta_0,\dots,\theta_n)}.$$
	(ii) Let $(\varphi_i^t)_{t\geq 0}$ be a smooth ray of smooth functions, $\Phi_i:X\times \mathbb{D}^*\to \mathbb{R},\Phi_i(x,e^{-t+\sqrt{-1}s})=\varphi_i^t(x)$. Then\begin{align*}
		dd^c_\tau\braket{\varphi_0^t,\dots,\varphi_n^t}_{(\theta_0,\dots,\theta_n)}=\int_X\bigwedge_{i=0}^n(\operatorname{pr}_1^*\theta_i+dd^c_{(x,\tau)}\Phi_i),
	\end{align*}where $\tau=e^{-t+\sqrt{-1}s}$, $\operatorname{pr}_1:X\times \mathbb{D}^*\to X$ is projection.
\end{proposition}
In the following we need a modified version of such pairing, {given by \cite{Dervan2023},} which can be understood as a generalization of $\braket{0,\cdots,0,\varphi_k,\cdots,\varphi_n}_{\braket{\theta_0,\cdots,\theta_n}}$.
\begin{definition}[\text{\citenum{Dervan2023}, Deﬁnition 3.1}]
	Let $k\in\mathbb{N}$, $\theta_k,\dots,\theta_n$ be closed $(1,1)$-forms, $B$ is a closed $(k,k)$-form on $X$.
	The $n+1-k$ ordered modified multivariate energy functional
	$\braket{\cdot,\dots,\cdot}_{(B)(\theta_k,\ldots,\theta_n)}$  
	on the space 
	$\mathcal{C}^\infty(X)
	\times\dots\times 
	\mathcal{C}^\infty(X)$  
	($n+1-k$ times) is defined by
		\begin{align}\label{modified}
			\langle \varphi_k,\dots,\varphi_n\rangle_{(B)(\theta_k,\dots,\theta_n)}
			&:= 
			\int_X \varphi_k\,B\wedge(\theta_{k+1} + dd^c\varphi_{k+1})\wedge\dots\wedge(\theta_n + dd^c\varphi_n)\notag\\ &+ \int_X \varphi_{k+1}\,B\wedge\theta_k\wedge(\theta_{k+2} + dd^c\varphi_{k+2})\wedge\dots\wedge(\theta_n + dd^c\varphi_n)
			\notag\\ &+ \int_X \varphi_{k+2}\,B\wedge\theta_k\wedge\theta_{k+1}(\theta_{k+3} + dd^c\varphi_{k+3})\wedge\dots\wedge(\theta_n + dd^c\varphi_n)
			\notag\\& + \dots
			+ \int_X \varphi_n\,B\wedge\theta_k\wedge\dots\wedge\theta_{n-1}.
		\end{align}
\end{definition}\noindent One can show that it has analogue properties as Proposition \ref{duichen} by using $dd^c$-lemma and integration by parts. In particular\begin{align}\label{my1}
	dd^c_\tau\braket{\varphi_k^t,\dots,\varphi_n^t}_{(B)(\theta_k,\dots,\theta_n)}=\int_X\operatorname{pr}_1^*B\wedge\bigwedge_{i=k}^n(\operatorname{pr}_1^*\theta_i+dd^c_{(x,\tau)}\Phi_i),
\end{align}{see \cite[Proposition 3.2]{Dervan2023}.} Note that if $B=\theta_0\wedge\cdots\wedge\theta_k$,  i.e. it can be decomposed into the product of $(1,1)$-forms, then we can get $\braket{\varphi_k,\dots,\varphi_n}_{(B)(\theta_k,\ldots,\theta_n)}=\braket{0,\cdots,0,\varphi_k,\cdots,\varphi_n}_{\braket{\theta_0,\cdots,\theta_n}}$ easily.\begin{remark}
	The modified multivariate energy functional  \eqref{my1} can also defined on $\operatorname{PSH}(X,\theta_k)\cap L^\infty_{\text{loc}}\times \cdots\times \operatorname{PSH}(X,\theta_k)\cap L^\infty_{\text{loc}}$ by the same formula. Note that if $\varphi_j\in \operatorname{PSH}(X,\theta_j)\cap L^\infty_{\text{loc}},k\leq j\leq n$, then we can interpret $\bigwedge_{j=k}^n(\theta_j+dd^c\varphi_j)$ as a closed positive current in the sense of Bedford-Taylor \cite{BedfordTaylor76,BedfordTaylor82}.
\end{remark}
\begin{remark}\label{remark9}
	The importance of the multivariate energy functional \eqref{delignepairing} is that it can reinterpret some functionals of the cscK equation theory, for example, the Monge-Amp\`{e}re functional and Aubin J-functional\begin{align*}
		\operatorname{E}(\varphi)&:=\frac{1}{n+1}\sum_{j=0}^n\int_X\varphi(\omega+dd^c\varphi)^j\wedge\omega^{n-j}=\frac{1}{n+1}\braket{\varphi,\dots,\varphi}_{(\omega,\dots,\omega)},\\
		\operatorname{J}(\varphi)&:=\int_{X}\varphi\omega^n-\operatorname{E}(\varphi)=\braket{\varphi,0,\dots,0}_{(\omega,\dots,\omega)}-\frac{1}{n+1}\braket{\varphi,\dots,\varphi}_{(\omega,\dots,\omega)}.
	\end{align*}In particular it can give an expansion for the Mabuchi K-energy $\operatorname{M}^{\operatorname{cscK}}$ by the Chen-Tian formula, i.e.\begin{align*}
		&\operatorname{M}^{\operatorname{cscK}}(\varphi) = \bar{S}\operatorname{E}(\varphi) - \operatorname{E}^{\mathrm{Ric}(\omega)}(\varphi) + \int_X \log \left( \frac{(\omega + dd^c \varphi)^n}{\omega^n} \right) (\omega + dd^c \varphi)^n\\ =&\  \frac{\bar{S}}{n+1}\braket{\varphi,\dots,\varphi}_{(\omega,\dots,\omega)}-\braket{0,\varphi,\dots,\varphi}_{(\operatorname{Ric}(\omega),\omega,\dots,\omega)}+\int_X \log \left( \frac{(\omega + dd^c \varphi)^n}{\omega^n} \right) (\omega + dd^c \varphi)^n,
	\end{align*}where $\bar{S}$ is the mean scalar curvature. Thus, for energy part $ \bar{S}\operatorname{E}(\varphi) - \operatorname{E}^{\mathrm{Ric}(\omega)}(\varphi)$ at least, one should expect a unified approach to computing the slope of the analytic Deligne pairing, treating the energy part's slope as a special case. This expectation is realised by \cite[Theorem 4.9]{sjostrom2018k}.
	
\end{remark}

\boldmath
\subsection{Algebraic test configurations, cohomological test configurations and smooth $\mathcal{H}_{\omega}$-rays compatible with them}
\unboldmath
\label{subsec:3}
As with Remark \ref{remark9}, we want to relate the $\alpha$-K-energy in Theorem \ref{alphaKenergy} to the modified multivariate energy functional (\ref{modified}), then to apply \cite[Theorem 4.9]{sjostrom2018k} to it. To explain this process, let us recall the definition of test configuration, which is originally given by Donaldson \cite{donaldson2002scalar}. We will follow \cite{boucksom2019} for the algebraic case and \cite{sjostrom2018k} for a generalization to the cohomological case. For convenience we will use the compactified version.

\begin{definition}Let $(V,L)$ be a smooth projective complex variety coupled with a line bundle, an ample algebraic test configuration $(\mathcal{V},\mathcal{L})$ for $(V,L)$ consists of the following data:\\
	(i) a flat, projective morphism of normal schemes $\pi: \mathcal{V} \to \mathbb{P}^1$ which is $\mathbb{C}^*$-equivariant;\\
	(ii) a $\mathbb{C}^*$-equivariant relatively ample $\mathbb{Q}$-line bundle $\mathcal{L}$;\\
	(iii) $(\mathcal{V}_t,\mathcal{L}_t) \simeq (V,L)$ for $t\neq 0$.
\end{definition}

\begin{definition}[\text{\citenum{sjostrom2018k}, Definition 3.2, 3.3 \& 3.4}]
	Let $X$ be a compact K\"{a}hler manifold and $\alpha\in H^{1,1}(X,\mathbb{R})$ be a cohomology class on $X$. A cohomological test configuration $(\mathcal{X},\mathcal{A})$ of $(X,\alpha)$ consists of:\\
	(i) a normal compact K\"{a}hler complex space $\mathcal{X}$ with a flat morphism $\pi:\mathcal{X}\to\mathbb{P}^1$ which is $\mathbb{C}^*$-equivariant;\\
	(ii) a $\mathbb{C}^*$-equivariant isomorphism $\mathcal{X}\backslash\mathcal{X}_0\simeq X\times(\mathbb{P}^1\backslash\{0\})$;\\
	(iii) a $\mathbb{C}^*$-invariant Bott-Chern cohomology class $\mathcal{A}\in H^{1,1}_{\operatorname{Bott-Chern}}(\mathcal X,\mathbb R)$ such that the pullback of $\mathcal{A}$ under the isomorphism $\mathcal{X}\backslash\mathcal{X}_0\simeq X\times(\mathbb{P}^1\backslash\{0\})$ is $\operatorname{pr}_1^*\alpha$.
	
	We say that a cohomological test configuration $(\mathcal{X}, \mathcal{A})$ for $(X, \alpha)$ is smooth
	if the total space $\mathcal{X}$ is smooth. In case $\alpha \in H^{1,1}(X,\mathbb{R})$ is Kähler,
	we say that $(\mathcal{X}, \mathcal{A})$ is relatively Kähler if the cohomology class 
	$\mathcal{A}$ is relatively Kähler, i.e.\ there is a Kähler form $\beta$ on $\mathbb{P}^1$ such that $\mathcal{A} + \pi^{*}\beta$  is Kähler on $\mathcal{X}$.
\end{definition}
A normal algebraic test configuration $(\mathcal{X},\mathcal{L})$ canonically induces a cohomological one $(\mathcal{X},c_1(\mathcal{L}))$.  In the present paper, we will apply the conclusion of the cohomological test configuration to the algebraic case to get Theorem \ref{my}.

Consider now a smooth algebraic test configuration $(\mathcal{X},\mathcal{L})$ for a polarized complex manifold $(X,L)$, which dominates $X\times \mathbb{P}^1$ with birational morphism $\mu$. By the normality of $\mathcal{X}$, we know that $\mathcal{L}=\mu^*p_1^*L+D$ for some $\mathbb{Q}$-Cartier divisor $D$ supported on $\mathcal{X}_0$. Suppose now that there is a smooth  metric $\Phi$ on $\mathcal{L}$ near the centre fiber. At the expense of some information outside $\mathcal{X}_0$, we can extend this metric to the entire space.  Then, we have the following equality of smooth Chern forms\begin{align*}
	c_1(\mathcal{L},\Phi)=\mu^*p_1^*c_1(L,\phi_1)+c_1(\mathcal{O}(D),h),
\end{align*}where $h$ is a $\mathbb{S}^1$-invariant metric  on the line bundle $\mathcal{O}(D)$ and $\phi_1=\Phi|_{\mathcal{X}_1}$. By Poincar\'{e}-Lelong equation, $c_1(\mathcal{O}(D),h)=\delta_D-dd^c\psi_D$ for some $\mathbb{S}^1$-invariant singular function $\psi_D\in\mathcal{C}^\infty(\mathcal{X}\backslash\mathcal{X}_0,\mathbb{R})$. Thus we have\begin{align*}
	c_1(\mathcal{L},\Phi)=\mu^*p_1^*c_1(L,\phi_1)-dd^c\psi_D
\end{align*}as current and therefore as smooth form on $\mathcal{X}\backslash\mathcal{X}_0$.

This means $\sqrt{-1}dd^c(\Phi_D|_{\mathcal{X}_\tau})=\sqrt{-1}dd^c\phi_1-dd^c(\psi_D|_{\mathcal{X}_\tau})$. Especially when the closed real $(1,1)$-forms $\omega_\tau=\sqrt{-1}dd^c(\Phi_D|_{\mathcal{X}_\tau})$ and $\omega_1=\sqrt{-1}dd^c\phi_1$ are positive, this gives a smooth ray in $\mathcal{H}_{\omega_1}$. When we further require $dd^c\Phi$ to be positive, the ray $-\psi_D$ becomes a subgeodesic ray. \cite{sjostrom2018k} extends such a relation between test configurations and $\mathcal{H}_\omega$-rays to the cohomological case, namely
\begin{definition}[\text{\citenum{sjostrom2018k}, Section 4.1, 4.2 and Lemma 4.10}]\label{compatible}
	Let $(\mathcal{X}, \mathcal{A})$ be a cohomological test configuration of $(X,\alpha)$ which is smooth and dominates $X\times \mathbb{P}^1$ with bimeromorphic morphism $\mu$, and $(\varphi_t)_{t\geq 0}$ be a smooth ray of smooth functions on $X$. Choose a smooth $S^1$-invariant $(1,1)$-form $\Theta$ representing $\mathcal{A}$, a smooth $(1,1)$-form $\theta$ representing $\alpha$ and an $S^1$-invariant smooth function $\Phi:X\times (\mathbb{P}\backslash \{0\})\to\mathbb{R}$ such that $\Phi(x,e^{-t+\sqrt{-1}s})=\varphi_t(x)$. We say that $(\varphi_t)_{t\geq 0}$ is compatible with $(\mathcal{X},\mathcal{A})$ if there exists a smooth $S^1$-invariant function $\Psi$ on $\mathcal{X}$ such that the current $\Theta+dd^c\Psi$ equals $\mu^*\operatorname{pr}_1^*\theta+\mu^*dd^c\Phi$ when restricted to $\mathcal{X}\backslash\mathcal{X}_0$.
\end{definition}
{It is clear that the algebraic test configuration and its induced $\mathcal{H}_\omega$-ray meet this definition.} Note that the definition is independent of the choice of 
$\Phi$, $\Theta$, and $\theta$, since replacing them by 
$\Phi'$, $\Theta + dd^c h$, and $\theta + dd^c g$ respectively (with $\Phi',h$ being $S^1$-invariant) 
changes $\Psi$ only by the addition of a smooth function of the form 
$\Phi' - \Phi + g \circ \operatorname{pr}_1 \circ \mu - h$.
\begin{definition}[\text{\citenum{sjostrom2018k}, Deﬁnition 2.5}]\label{weakcompatible}
	A smooth ray $(\varphi_t)_{t\geq 0}$ of functions in $\mathcal{H}_\omega$ is called a subgeodesic $\mathcal{H}_\omega$-ray if $\Phi$ is $\operatorname{pr}_1^*\omega$-psh on $X\times\mathbb{D}^*$, where $\Phi(x,e^{-t+\sqrt{-1}s})=\varphi_t(x)$.
\end{definition}
\begin{lemma}[\text{\citenum{sjostrom2018k}, Lemma 4.4}]
	If $(\mathcal{X},\mathcal{A})$ is a smooth, relatively Kähler cohomological test configuration of $(X,\omega)$, then a subgeodesic $\mathcal{H}_\omega$-ray that is $\mathcal{C}^\infty$-compatible with $(\mathcal{X},\mathcal{A})$ always exists.
\end{lemma}

\subsection{Asymptotic slope of modified multivariate energy functionals}
\label{subsec:4}
Now we introduce a general version of \cite[Theorem 4.9]{sjostrom2018k}, which has been stated in \cite{Dervan2023}, and we apply it to the modified multivariate energy functional (\ref{modified}).

As in \cite[Section 3.1]{sjostrom2018k}, we say that the test configuration $\mathcal{X}$ dominates another $\mathcal{X}'$, if the meromorphic map between test configurations $\mathcal{X}\dashrightarrow\mathcal{X}'$ (induced by $\mathcal{X}\backslash\mathcal{X}_0\simeq X\times(\mathbb{P}^1\backslash\{0\})\simeq\mathcal{X}'\backslash\mathcal{X}'_0$) is a $\mathbb{C}^*$-equivariant morphism (see also \cite[Section 2]{boucksom2017uniform}). For any finite number of test configurations $\mathcal{X}_1,...,\mathcal{X}_k$, by equivariant resolution of singularities, there exists a smooth test configuration $\mathcal{X}$ dominates all $\mathcal{X}_i$ and $X\times\mathbb{P}^1$ (c.f.  \cite[Lemma 4.1.1]{dyrefelt2017k}). Thus we can define the intersection number\begin{align*}
	(\mathcal{A}_0\cdot\dots\cdot\mathcal{A}_n),
\end{align*}for any test configurations $(\mathcal{X}_i,\mathcal{A}_i),i=0,...,n$, by pulling back the respective cohomology classes $\mathcal{A}_i$ to $\mathcal{X}$ that dominates all $\mathcal{X}_i$. By the projection formula we know such intersection number is independent of the choice of common resolution $\mathcal{X}$. The similar holds true for the algebraic case, see \cite[Section 6.6]{boucksom2017uniform}.
\begin{lemma}[\text{\citenum{Dervan2023}, Lemma 3.2}]
	Let $X$ be a compact K\"{a}hler manifold of dimension $n$, {$0\leq k<n$ an integer,} $B$ a closed $(k,k)$-form and $\theta_i,i=k,...,n$ closed $(1,1)$-forms on $X$, and let $[\theta_i]$ be the cohomology class of $\theta_i$ for $i=k,...,n$. Consider smooth cohomological test conﬁgurations $(\mathcal{X}_i,\mathcal{A}_i)$ for $(X,[\theta_i])$ dominating $X\times\mathbb{P}^1$. For each collection of smooth rays $(\varphi_i^t)_{t\geq 0}$ of smooth functions that $\mathcal{C}^\infty$-compatible with $(\mathcal{X}_i,\mathcal{A}_i),i=k,...,n$, we have\begin{align*}
		\lim_{t\to\infty}\frac{\braket{\varphi_k^t,\dots,\varphi_n^t}_{(B)(\theta_k,\ldots,\theta_n)}}{t}=(\mathcal{B}\cdot \mathcal{A}_k\cdot\dots\cdot \mathcal{A}_n),
	\end{align*}by pulling back the cohomology classes $\mathcal{B}=[\operatorname{pr}_1^*B]$ (on trivial test configuration $X\times\mathbb{P}^1$) and $\mathcal{A}_i$ to smooth test configuration $\mathcal{X}$ that dominates all $\mathcal{X}_i$ and $X\times\mathbb{P}^1$.
\end{lemma}
\begin{proof}
	After pulling back we can assume all $\mathcal{X}_i$ are equal to $\mathcal{X}$, with $\mu:\mathcal{X}\to X\times \mathbb{P}^1$ being the dominating map. Let $\Phi_i:X\times(\mathbb{P}\backslash\{0\})\to \mathbb{R}$ be a smooth function such that $\Phi_i(x,e^{-t+\sqrt{-1}s})=\varphi_i^t(x)$ and choose a smooth $S^1$-invariant $(1,1)$-form $\Theta_i$ representing $\mathcal{A}_i$. By Definition \ref{compatible}, $\varphi_i^t$ compatible with $(\mathcal{X},\mathcal{A}_i)$ means that there exists a smooth $S^1$-invariant function $\Psi_i$ on $\mathcal{X}$ such that the current $\Theta_i+dd^c\Psi_i$ equals $\mu^*\operatorname{pr}_1^*\theta_i+\mu^*dd^c\Phi_i$ when restricted to $\mathcal{X}\backslash \mathcal{X}_0$. Set $u(\tau):=\braket{\varphi_k^t,\dots,\varphi_n^t}_{(B)(\theta_k,\ldots,\theta_n)}$, and then we have\begin{small}
		\begin{align*}
			&\frac{d}{dt}\bigg|_{t=-\log \varepsilon}u(\tau)=\int_{\mathbb{P}^1\backslash\varepsilon\mathbb{D}}dd^c_\tau\braket{\varphi_k^t,\dots,\varphi_n^t}_{(B)(\theta_k,\ldots,\theta_n)}\\ \overset{(\ref{my1})}{=}\ &\int_{X\times (\mathbb{P}^1\backslash\varepsilon\mathbb{D})}\operatorname{pr}_1^*B\wedge\bigwedge_{i=k}^n(\operatorname{pr}_1^*\theta_i+dd^c_{(x,\tau)}\Phi_i)\\=\ & \int_{\pi^{-1}(\mathbb{P}\backslash\varepsilon\mathbb{D})}\mu^*\operatorname{pr}_1^*B\wedge\bigwedge_{i=k}^n(\mu^*\operatorname{pr}_1^*\theta_i+\mu^*dd^c_{(x,\tau)}\Phi_i)\\= \  &\int_{\pi^{-1}(\mathbb{P}\backslash\varepsilon\mathbb{D})}\mu^*\operatorname{pr}_1^*B\wedge\bigwedge_{i=k}^n(\Theta_i+dd^c_{(x,\tau)}\Psi_i)\to (\mathcal{B}\cdot \mathcal{A}_k\cdot\dots\cdot \mathcal{A}_n),\quad \varepsilon\to 0, 
		\end{align*}
	\end{small}where $\tau=e^{-t+\sqrt{-1}s}\text{ and }\pi:\mathcal{X}\to \mathbb{P}^1$.
	
	Note that \begin{align*}
		u(\tau)&=\braket{\varphi_k^t,\dots,\varphi_n^t}_{(B)(\theta_k,\ldots,\theta_n)}=\braket{\varphi_k^t,\dots,\varphi_n^t}_{(M-M')(\theta_k,\ldots,\theta_n)}\\&=\braket{\varphi_k^t,\dots,\varphi_n^t}_{(M)(\theta_k,\ldots,\theta_n)}-\braket{\varphi_k^t,\dots,\varphi_n^t}_{(M')(\theta_k,\ldots,\theta_n)}
	\end{align*}for some positive $(k,k)$-forms $M,M'$. Continue this procedure,\begin{align*}
		&\braket{\varphi_k^t,\dots,\varphi_n^t}_{(M)(\theta_k,\ldots,\theta_n)}=\braket{\varphi_k^t,\dots,\varphi_n^t}_{(M)(\omega_k-\omega_k',\theta_{k+1},\ldots,\theta_n)}\\=&\ \braket{\varphi_k^t,\dots,\varphi_n^t}_{(M)(\omega_k,\theta_{k+1},\dots,\theta_n)}-\braket{0,\varphi_{k+1}^t,\dots,\varphi_n^t}_{(M)(\omega_k,\theta_{k+1},\dots,\theta_n)}
	\end{align*}for some positive $(1,1)$-forms $\omega_k,\omega_k'$. Thus by the symmetry in Proposition \ref{duichen}, $u(\tau)$ can be decomposed into a diﬀerence of convex functions, so\begin{align*}
		\lim_{t\to \infty}\frac{u(\tau)}{t}=\lim_{\varepsilon\to 0}\frac{d}{dt}\bigg|_{t=-\log \varepsilon}u(\tau)=(\mathcal{B}\cdot \mathcal{A}_k\cdot\dots\cdot \mathcal{A}_n),
	\end{align*}and the theorem follows. 
\end{proof}

\boldmath
\section{Expanding the $\alpha$-K-energy}
\unboldmath
\label{sec:4}

Let $E$ be an irreducible complex vector bundle of rank $r$ over a compact K\"{a}hler manifold $(X,\omega)$, consider the holomorphic frame bundle $P=F(E)$ of $E$, which is a principal $\operatorname{GL}_r(\mathbb{C})$-bundle. Just like section \ref{sec:2}, we fix inner product on $\mathfrak{u}(r)$ by $-\operatorname{tr}(\cdot\circ\cdot)$, which extends to a symmetric $\mathbb{C}$-bilinear product on $\mathfrak{u}(r)^c=\mathfrak{gl}(r,\mathbb{C})$.

At this moment, the space $\mathcal{R}$ of all $U(r)$-reduction of $P=F(E)$ is just $\operatorname{Herm}^+(E)$, the space of all Hermitian metrics on $E$. For $z$ in \eqref{KYM}, after a simple computation by \cite[(1.18)]{alvarez2013coupled}, we have $z=-2\pi\sqrt{-1}\frac{\mu(E)}{\operatorname{vol}_\omega}\operatorname{id}_E$ (under isomorphism of $\operatorname{End}E\simeq \operatorname{ad}P$), and thus\begin{align*}
	\Lambda_{\omega}(F_{H}\wedge z)\frac{\omega^n}{n!}=-\mu'(E)\Lambda_\omega(\operatorname{tr}(R_{H}))\frac{\omega^n}{n!}=-\mu'(E)\operatorname{tr}(R_{H})\wedge\frac{\omega^{n-1}}{(n-1)!},
\end{align*}where $\mu'(E):=-2\pi\sqrt{-1}\frac{\mu(E)}{\operatorname{vol}_\omega}$, $R_{H}\in\Omega^2(\operatorname{End}E)$ is the Chern curvature form of hermitian metric $H$ on the vector bundle $E$. Note that the irreducibility of $E$ is essential here; otherwise, $z$ could be written as a sum of scalar multiples of the identity on holomorphic subbundles, and consequently, $z \wedge F_H$ would also decompose as a sum.

By Corollary \ref{independent}, we can compute the formula $$4\operatorname{M}(b_t,b_0)+4Q_2(b_t,b_0)-Q_1(b_t,b_1)$$ by choosing a special line $\gamma_s^{(t)}$\begin{align}\label{gamma}
	b_s^{(t)}:=(s\varphi_t,H_0),s\in [0,1]\  \text{connected behind}\ b_s^{[t]}:=(\varphi_t,H_s),s\in [0,t],
\end{align}where $\varphi_t$ is smooth function that satisfies $\omega_t-\omega_0=dd^c\varphi_t$. As in \cite[Section 3]{chen2000lower}, we have\begin{align*}
	&4\operatorname{M}(b_t,b_0)+4Q_2(b_t,b_0)-Q_1(b_t,b_1)=4\operatorname{M}'(b_t,b_0)+4Q_2'(b_t,b_0)-Q_1'(b_t,b_0),\\
	&\begin{aligned}
		&Q_1'(b_t,b_0)=\frac{1}{2\pi}\int_{0}^{1}\int_X(s\varphi_t)'|_{s}\left(\Lambda_{\omega_s^{(t)}}^2\left(F_{H_0}\wedge F_{H_0}\right)-2\pi\overline{C}_1\right)\frac{\left(\omega_s^{(t)}\right)^n}{n!}\wedge ds\\&=-\frac{1}{2\pi}\int_{0}^{1}\int_X\varphi_t\operatorname{tr}(R_{0}\wedge R_{0})\wedge\frac{\left(\omega_s^{(t)}\right)^{n-2}}{(n-2)!}\wedge ds-\overline{C}_1\int_{0}^{1}\int_X\varphi_t\frac{\left(\omega_s^{(t)}\right)^n}{n!}\wedge ds,
	\end{aligned}\\
	&\begin{aligned}
		&Q_2'(b_t,b_0)=\frac{1}{2\pi}\int_{0}^{1}\int_X(s\varphi_t)'|_{s}\left(\Lambda_{\omega_s^{(t)}}F_{H_0}\wedge z-2\pi\overline{C}_2\right)\frac{\left(\omega_s^{(t)}\right)^n}{n!}\wedge ds\\&=-\frac{\mu'(E)}{2\pi}\int_{0}^{1}\int_X\varphi_t\operatorname{tr}\left(R_{H_0}\right)\wedge\frac{\left(\omega_s^{(t)}\right)^{n-1}}{(n-1)!}\wedge ds-\overline{C}_2\int_{0}^{1}\int_X\varphi_t\frac{\left(\omega_s^{(t)}\right)^n}{n!}\wedge ds,
	\end{aligned}\\
	&\begin{aligned}
		\operatorname{M}'(b_t,b_0)=\int_{0}^{t}\int_X\sqrt{-1}\dot{H}_s\wedge \left(\Lambda_{\omega_t}F_{H_s}-z\right)\frac{\left(\omega_t\right)^n}{n!}\wedge ds,
	\end{aligned}
\end{align*}here $\omega_s^{(t)}=\omega_0+sdd^c\varphi_t$.

Note that $Q_1,Q_2,\operatorname{M}$ may depend on the path but $Q_1',Q_2'$ not. In fact, $\operatorname{M}'$ does not depend on the path either, see Section \ref{subsec:6}.

\boldmath
\subsection{Multivariate functional for $Q_1'$ and $Q_2'$}
\unboldmath
\label{subsec:5}

Let's compute the common factor $\int_{0}^{1}\int_X\varphi_t\left(\omega_s^{(t)}\right)^n\wedge ds$ of $Q_1'$ and $Q_2'$ firstly, which is actually $\operatorname{E}(\varphi)$ by \cite{chen2000lower}, and thus equal to $\frac{1}{n+1}\braket{\varphi_t,...,\varphi_t}_{(\omega_0,...,\omega_0)}$. For the reader's convenience, we provide a proof here:
\begin{align*}
	&\int_{0}^{1}\int_X\varphi_t\left(\omega_s^{(t)}\right)^n\wedge ds\\=&\ \int_{0}^{1}\int_X\varphi_t\left(\omega_0+sdd^c\varphi_t\right)^n\wedge ds\\ =&\ \sum_{p=0}^{n}\int_{0}^{1}\int_X\varphi_ts^{n-p}C_n^p\omega_0^p\wedge(dd^c\varphi_t)^{n-p}\wedge ds\\ = &\ \sum_{p=0}^{n}\frac{1}{n-p+1}\int_X\varphi_tC_n^p\omega_0^p\wedge(dd^c\varphi_t)^{n-p}\\ = &\ \sum_{p=0}^{n}\frac{n!}{(n-p+1)!p!}\int_X\varphi_tC_n^p\omega_0^p\wedge(dd^c\varphi_t)^{n-p}\\= &\ \frac{1}{n+1}\sum_{p=0}^{n}C_{n+1}^p\int_X\varphi_t\omega_0^p\wedge(dd^c\varphi_t)^{n-p}
\end{align*}which equals to $\frac{1}{n+1}\braket{\varphi_t,...,\varphi_t}_{(\omega_0,...,\omega_0)}$ since\begin{align*}
	&\braket{\varphi_t,...,\varphi_t}_{(\omega_0,...,\omega_0)}=\sum_{q=0}^{n}\int_X\varphi_t\omega_0^q\wedge (\omega_0+dd^c\varphi_t)^{n-q}\\=&\ \sum_{q=0}^{n}\sum_{i=0}^{n-q}\int_X\varphi_t C_{n-q}^i\omega_0^q\wedge \omega_0^i\wedge(dd^c\varphi_t)^{n-q-i}\\ =& \ \sum_{p=0}^n\sum_{i+q=p}\int_X\varphi_t C_{n-q}^i\omega_0^p\wedge(dd^c\varphi_t)^{n-p}\\ = &\ \sum_{p=0}^n\sum_{i=0}^p\int_X\varphi_t C_{n-p+i}^i\omega_0^p\wedge(dd^c\varphi_t)^{n-p}\\=&\ \sum_{p=0}^n\sum_{i=0}^p\int_X\varphi_t \left(C_{n-p+i+1}^i-C_{n-p+i}^{i-1}\right)\omega_0^p\wedge(dd^c\varphi_t)^{n-p}\\ =&\ \sum_{p=0}^n\sum_{i=0}^p\int_X\varphi_tC_{n+1}^p\omega_0^p\wedge(dd^c\varphi_t)^{n-p}.
\end{align*}

For $Q_2'$, up to a scalar $-\frac{\mu'(E)}{2\pi(n-1)!}$ for simply, we can just use a computation like \cite{chen2000lower}\begin{align*}
	&\int_{0}^{1}\int_X\varphi_t\operatorname{tr}(R_0)\wedge\left(\omega_s^{(t)}\right)^{n-1}\wedge ds\\ =\ & \int_{0}^{1}\int_X\varphi_t\operatorname{tr}(R_0)\wedge\left(\omega_0+sdd^c\varphi_t\right)^{n-1}\wedge ds\\= \ &\sum_{p=0}^{n-1}\int_{0}^{1}\int_Xs^{n-1-p}\varphi_t\operatorname{tr}(R_0)\wedge C_{n-1}^p\omega_0^{p}\wedge (dd^c\varphi_t)^{n-1-p}\wedge ds\\= \ &\sum_{p=0}^{n-1}\int_X\frac{(n-1)!}{p!(n-p)!}\varphi_t\operatorname{tr}(R_0)\wedge\omega_0^{p}\wedge (dd^c\varphi_t)^{n-1-p}\\=\ &\frac{1}{n}\sum_{p=0}^{n-1}\int_XC_n^p\varphi_t\operatorname{tr}(R_0)\wedge\omega_0^{p}\wedge (dd^c\varphi_t)^{n-1-p}.
\end{align*}On the other hand, the Deligne pairing (\ref{delignepairing}) gives us
\begin{align*}
	&\braket{0,\varphi_t,\varphi_t,...,\varphi_t}_{(\operatorname{tr}R_0,\omega_0,...,\omega_0)}=\sum_{q=0}^{n-1}\int_X\varphi_t\operatorname{tr}(R_0)\wedge \omega_0^q\wedge (\omega_0+dd^c\varphi_t)^{n-1-q}\nonumber\\=\ & \sum_{q=0}^{n-1}\sum_{i=0}^{n-1-q}\int_X\varphi_tC_{n-1-q}^i\operatorname{tr}(R_0)\wedge \omega_0^q\wedge \omega_0^i\wedge (dd^c\varphi_t)^{n-1-q-i}\nonumber\\ =\ &\sum_{p=0}^{n-1}\sum_{i+q=p}\int_X\varphi_tC_{n-1-q}^i\operatorname{tr}(R_0)\wedge \omega_0^p\wedge (dd^c\varphi_t)^{n-1-p}\nonumber\\=\ &\sum_{p=0}^{n-1}\sum_{i=0}^{p}\int_X\varphi_tC_{n-1-p+i}^i\operatorname{tr}(R_0)\wedge \omega_0^p\wedge (dd^c\varphi_t)^{n-1-p}\nonumber\\=\ &\sum_{p=0}^{n-1}\sum_{i=0}^{p}\int_X\varphi_tC_{n-1-p+i}^i\operatorname{tr}(R_0)\wedge \omega_0^p\wedge (dd^c\varphi_t)^{n-1-p}\nonumber\\ =\ &\sum_{p=0}^{n-1}\sum_{i=0}^{p}\int_X\varphi_t\left(C_{n-1-p+i+1}^i-C_{n-1-p+i}^{i-1}\right)\operatorname{tr}(R_0)\wedge \omega_0^p\wedge (dd^c\varphi_t)^{n-1-p}\nonumber\\=\ &\sum_{p=0}^{n-1}\int_X\varphi_tC_{n}^p\operatorname{tr}(R_0)\wedge \omega_0^p\wedge (dd^c\varphi_t)^{n-1-p},
\end{align*}and thus we have the following identity\begin{align}\label{my2}
	\int_{0}^{1}\int_X\varphi_t\operatorname{tr}(R_0)\wedge\left(\omega_s^{(t)}\right)^{n-1}\wedge ds=\frac{1}{n}\braket{0,\varphi_t,\varphi_t,...,\varphi_t}_{(\operatorname{tr}R_0,\omega_0,...,\omega_0)}.
\end{align}

Similarly for $Q_1'$ it suffices to compute (up to a scalar $-\frac{1}{2\pi(n-2)!}$)\begin{align*}
	&\int_{0}^{1}\int_X\varphi_t\operatorname{tr}(R_0\wedge R_0)\wedge\left(\omega_s^{(t)}\right)^{n-2}\wedge ds\\ =\ &\int_{0}^{1}\int_X\varphi_t\operatorname{tr}(R_0\wedge R_0)\wedge\left(\omega_0+sdd^c\varphi_t\right)^{n-2}\wedge ds\\=\ &\sum_{p=0}^{n-2}\int_{0}^{1}\int_Xs^{n-2-p}C_{n-2-p}^p\varphi_t\operatorname{tr}(R_0\wedge R_0)\wedge \omega_0^p\wedge (dd^c\varphi_t)^{n-2-p}\wedge ds\\=\ &\sum_{p=0}^{n-2}\int_X\frac{(n-2)!}{p!(n-1-p)!}\varphi_t\operatorname{tr}(R_0\wedge R_0)\wedge \omega_0^p\wedge (dd^c\varphi_t)^{n-2-p}\\=\ & \frac{1}{n-1}\sum_{p=0}^{n-2}\int_XC_{n-1}^p\varphi_t\operatorname{tr}(R_0\wedge R_0)\wedge \omega_0^p\wedge (dd^c\varphi_t)^{n-2-p}.
\end{align*}On the another hand, the $n-1$ ordered modified Deligne pairing (\ref{modified}) gives us
\begin{align*}
	&\braket{\varphi_t,\dots,\varphi_t}_{(\operatorname{tr}(R_0\wedge R_0))(\omega_0,\dots,\omega_0)}\\  = \ &\sum_{q=0}^{n-2}\int_X\varphi_t\operatorname{tr}(R_0\wedge R_0)\wedge\omega_0^q\wedge(\omega_0+dd^c\varphi_t)^{n-2-q}\nonumber\\=\ &\sum_{q=0}^{n-2}\sum_{i=0}^{n-2-q}\int_XC_{n-2-q}^i\varphi_t\operatorname{tr}(R_0\wedge R_0)\wedge\omega_0^{q+i}\wedge (dd^c\varphi_t)^{n-2-q-i}\nonumber\\= \ &\sum_{p=0}^{n-2}\sum_{i+q=p}\int_XC_{n-2-q}^i\varphi_t\operatorname{tr}(R_0\wedge R_0)\wedge\omega_0^{p}\wedge (dd^c\varphi_t)^{n-2-p}\nonumber\\= & \ \sum_{p=0}^{n-2}\sum_{i=0}^p\int_XC_{n-2-p+i}^i\varphi_t\operatorname{tr}(R_0\wedge R_0)\wedge\omega_0^{p}\wedge (dd^c\varphi_t)^{n-2-p}\nonumber\\ =  &\ \sum_{p=0}^{n-2}\sum_{i=0}^p\int_X(C_{n-2-p+i+1}^{i}-C_{n-2-p+i}^{i-1})\varphi_t\operatorname{tr}(R_0\wedge R_0)\wedge\omega_0^{p}\wedge (dd^c\varphi_t)^{n-2-p}\nonumber\\ = &\ \sum_{p=0}^{n-2}\int_XC_{n-1}^p\varphi_t\operatorname{tr}(R_0\wedge R_0)\wedge\omega_0^{p}\wedge (dd^c\varphi_t)^{n-2-p},
\end{align*}and thus we have\begin{align}\label{my3}
	\int_{0}^{1}\int_X\varphi_t\operatorname{tr}(R_0\wedge R_0)\wedge\left(\omega_s^{(t)}\right)^{n-2}\wedge ds=\frac{1}{n-1}\braket{\varphi_t,\dots,\varphi_t}_{(\operatorname{tr}(R_0\wedge R_0))(\omega_0,\dots,\omega_0)}.
\end{align}

\boldmath
\subsection{The functional $\operatorname{M}'$}
\unboldmath
\label{subsec:6}
For $\operatorname{M}'$ we need to unpack $\dot{H}_s$, which is described as a section of $\sqrt{-1}\operatorname{ad}P_{H_s}$ in \cite[(3.120)]{alvarez2013coupled}. Recall that $P_{H}$ is the $U(r)$-reduction of $P$ induced by the Hermitian metric $H$. 
A point of the adjoint bundle $\operatorname{ad} P_{H} = P_{H}\times_{U(r)}\mathfrak{u}(r)$ can be represented by an equivalence class 
\([e, Y(e)]\), 
where $e$ and $e'$ are $H$-orthonormal frames (written as row vectors) over the same point $x=\pi(e)=\pi(e')\in X$, $Y:P_{H}\to\mathfrak{u}(r)$ is a $\mathfrak{u}(r)$-valued function on $P_{H}$, 
and the equivalence relation is given by
\[
(e, Y(e)) \sim (e', Y(e')) 
\Leftrightarrow 
e' = eA,
Y(e') = \operatorname{Ad}_{A^{-1}} Y(e),\  \text{for some}\  A\in U(r).
\] Giving a section of $\operatorname{ad}P_{H_s}$ is equivalent to giving a endomorphism of $E$ whose matrix under any $H$-orthonormal basis is in $\mathfrak{u}(r)$. Actually there is an embedding\begin{align}\label{equivalentclass}
	\begin{aligned}
		\Gamma(\operatorname{ad}P_H)&\hookrightarrow \operatorname{End}E,\\
		[e,Y(e)]&\mapsto \varphi_Y: p=ev\mapsto eY(e)v,
	\end{aligned}
\end{align}where $v$ is a column vector. Note that $Y(eA)=A^{-1}Y(e)A$ and implies that $\varphi_Y(e'v')=e'Y(e')v'=eAA^{-1}Y(e)AA^{-1}v=eY(e)v=\varphi_Y(ev)$ if $ev=e'v'$.

Let $\operatorname{Herm}(E,H)$ be the space of all endomorphisms that can be represented by a positive definite Hermitian matrix under any $H$-orthonormal frame of $E$, and $\operatorname{Herm}^+(E)$ be the space of all Hermitian metrics, i.e. the subspace of $\operatorname{Hom}(E,\overline{E}^\vee)$ that contains all linear morphism $\varphi$ with matrix $(\braket{\varphi(e_i),\overline{e}_j})_{(i,j)}$ positive Hermitian. Then we have
\begin{proposition}
	For a fixed hermitian metric $H$, there are one-to-one correspondences\begin{align*}
		a\in\operatorname{Herm}(E,H)&\leftrightarrow h(a)\in \operatorname{Herm}^{+}(E) \\&\leftrightarrow s(a)\in \Gamma(\sqrt{-1}\operatorname{ad}P_H)\\&\leftrightarrow H(a)\in\Gamma(P/U(r)),
	\end{align*}where (i) $h(a)=H\circ\exp(a)$ is the combination of map $a\in \operatorname{End}(E)$ with $H\in \operatorname{Hom}(E,\overline{E}^\vee)$.\\
	(ii) Over a point $x\in X$, the value of the section $s(a)$ can be written as the class
	\[
	s(a)(x)=\left[e,-\tfrac{1}{2}a(e)\right],
	\]
	that is, the equivalence class determined by an $H$-orthonormal frame $e$ (viewed as a point of $P_H$) over $x$ together with the matrix $-\tfrac{1}{2}a(e)\in \sqrt{-1}\,\mathfrak{u}(r)$, where $a(e)$ denotes the matrix representing the endomorphism $a$ in the frame $e$.  
	Note that $\sqrt{-1}\,\mathfrak{u}(r)$ is precisely the space of Hermitian matrices.\\
	(iii) Over a point $x\in X$, the section $H(a)$ takes the form 
	\[
	H(a)(x) = \left[e \cdot \exp(-\tfrac{1}{2} a(e))\right],
	\] 
	i.e., the $U(r)$-equivalence class of point $e \cdot \exp(-\tfrac{1}{2} a(e))$ in $P$, for $H$-orthonormal frames $e$ over $x$. Recall that a point in the bundle $P/U(r)$ over $x \in X$ is a $U(r)$-equivalence class of frames of $E_x$, which can be regarded as an orthonormal basis for some Hermitian metric on $E_x$.
\end{proposition}
\begin{proof}
	The first correspondence is given by \cite[Section 6.2]{kobayashi2014differential}. The third one is from \cite[(3.120)]{alvarez2013coupled}. We only need to check that the $U(r)$-equivalent class $\left[e\cdot\exp\left(-\frac{1}{2}a(e)\right)\right]$ is orthonormal under Hermitian metric $h(a)$ over $x$, which gives the correspondence between $h(a)$ and $H(a)$.
	
	Choose an local $H$-orthonormal frame $e=(e_1,...,e_r)$ of $E$, Hermitian matrix $A$ is the matrix of $a\in\operatorname{End}(E)$ under $e$, let $e'=(e_1,...,e_r)\exp\left(-\frac{1}{2}A\right)=\left(\sum_ke_k\exp\left(-\frac{1}{2}A\right)_{ki}\right)$, then\begin{align*}
		\braket{h(a)e_i',\overline{e_j'}}&=\sum_{k,l}\braket{h(a)e_k\exp\left(-\frac{1}{2}A\right)_{ki},\overline{e_l\exp\left(-\frac{1}{2}A\right)_{lj}}}\\&=\sum_{k,l}\braket{h\circ\exp(a)\left(e_k\exp\left(-\frac{1}{2}A\right)_{ki}\right),\overline{e_l\exp\left(-\frac{1}{2}A\right)_{lj}}}\\&=\sum_{k,l}\braket{h\circ\exp(a)(e_k)\exp\left(-\frac{1}{2}A\right)_{ki},\overline{e_l\exp\left(-\frac{1}{2}A\right)_{lj}}}\\&=\sum_{k,l,m}\braket{h(e_m)\exp(A)_{mk}\exp\left(-\frac{1}{2}A\right)_{ki},\overline{e_l\exp\left(-\frac{1}{2}A\right)_{lj}}}\\&=\sum_{m,j}\exp\left(\frac{1}{2}A\right)_{mk}\overline{\exp\left(-\frac{1}{2}A\right)_{mj}}\\&=\sum_{m,j}\exp\left(\frac{1}{2}A\right)_{mk}\exp\left(-\frac{1}{2}A\right)_{jm}=\delta_{jk},
	\end{align*}which gives the proof.
\end{proof}
\begin{corollary}
	Let $H_s=H_0\circ \exp(a_s)$ be a smooth curve in $\mathcal{R}=\operatorname{Herm}^+(E)$ with $a_s\in \operatorname{Herm}(E,H_0)$, $U$ an open subset of $X$, $e$ a local $H_0$-orthonormal frame of $E$ on $U$ and $A_s$ the matrix representing the endomorphism $a_s$ in the frame $e$. Then the derivative $\dot{H}_s\in \sqrt{-1}\operatorname{ad}P_{H_0}$ equals to $[e,-1/2A_s']$ (equivalence class defined by (\ref{equivalentclass})) on $U$.
\end{corollary}
The minus in  $-1/2A_s'$ balances the minus in the adjoint-invariant product $-\operatorname{tr}(\cdot\circ\cdot)$. Recall that in \cite[(6.3.27)]{kobayashi2014differential}, the Donaldson functional is defined by
$$\operatorname{M}^{\operatorname{Don}}(H_1,H_0):=\int_{0}^{1}\int_X\operatorname{tr}\left(h_s^{-1}\partial_s h_s\cdot \left(\sqrt{-1}\Lambda_\omega R_s-2\pi\frac{\mu(E)}{\operatorname{vol}_\omega}\operatorname{id}_E\right)\right)\frac{\omega^n}{n!},$$
which is independent to the choice of path $H_s$. Here $h_s$ denotes the matrix representation of $H_s$ and $R_s$ is the curvature form of the associated Chern connection, both expressed in a fixed local frame of $E$. Note that different choices of local frame lead to conjugate matrices, and the trace ensures that the integrand is independent of this choice.

The Donaldson functional here has the same sign with $\operatorname{M}'$ since
\begin{align*}
	&\int_{0}^{1}\int_X\sqrt{-1}\dot{H}_s\wedge (\Lambda_{\omega}F_{H_s}-z)\frac{\omega^n}{n!}\wedge ds\\ =\ &\int_{0}^{1}\int_X\sqrt{-1}\left(-1/2A_s'\right)\wedge \left(\Lambda_{\omega}F_{H_s}-\left(-2\pi\sqrt{-1}\frac{\mu(E)}{\operatorname{vol}_\omega}\operatorname{id}_E\right)\right)\frac{\omega^n}{n!}\wedge ds
	\\ =\ &  \int_{0}^{1}\int_X-\sqrt{-1}\operatorname{tr}\left[\left(-1/2h_s^{-1}\partial_sh_s\right)\wedge \left(\Lambda_{\omega}R_{s}-\left(-2\pi\sqrt{-1}\frac{\mu(E)}{\operatorname{vol}_\omega}\operatorname{id}_E\right)\right)\right]\frac{\omega^n}{n!}\wedge ds \\ =\ & 1/2\int_{0}^{1}\int_X\operatorname{tr}\left[\left(h_s^{-1}\partial_sh_s\right)\wedge \left(\sqrt{-1}\Lambda_{\omega}R_{s}-2\pi\frac{\mu(E)}{\operatorname{vol}_\omega}\operatorname{id}_E\right)\right]\frac{\omega^n}{n!}\wedge ds.
\end{align*}In particular we have the following identity.
\begin{corollary}\label{mycorollary1}
	$\operatorname{M}'(b_t,b_0)=1/2\operatorname{M}^{\operatorname{Don}}_{\omega_t}(H_t,H_0)$, here $\operatorname{M}^{\operatorname{Don}}_{\omega_t}$ is the Donaldson functional with respect to the K\"{a}hler form $\omega_t$. Consequently $\operatorname{M}'$ is independent of the choice of path.
\end{corollary}
\begin{remark}
	The $1/2$ here is also natural. Because in \cite{alvarez2013coupled}, if the reader tries to compute the Gateau derivative of $\mathcal{M}_I$ directly and verify its closedness, this $1/2$ will balance the $2$ in $dd^c = 2\sqrt{-1}\partial\bar\partial$.
\end{remark}

\section{Proof of the main results}
\label{sec:5}
\subsection{Proof of (\ref{mytheorem1}) and Theorem \ref{my}}
\label{subsec:7}
We are now ready to prove our first result (\ref{mytheorem1}). For the curve $\gamma_s$ in (\ref{gamma}), by (\ref{my2}), (\ref{my3}) and Corollary \ref{mycorollary1}, we have 
\begin{align*}
	&\mathcal{M}_I(b_t,b_0)=2\alpha_1\operatorname{M}^{\operatorname{Don}}_{\omega_t}(h_t,h_0)+\alpha_0\operatorname{M}^{\operatorname{cscK}}(\omega_t,\omega_0)-\frac{\alpha_1}{2\pi}Q_1'+\frac{2\alpha_1}{\pi}Q_2'
	\\&=2\alpha_1\operatorname{M}^{\operatorname{Don}}_{\omega_t}(h_t,h_0)+\alpha_0\operatorname{M}^{\operatorname{cscK}}(\omega_t,\omega_0)-\frac{\alpha_1}{2\pi}Q_1'
	\\& \quad \quad +4\sqrt{-1}\alpha_1\frac{\mu(E)}{n!\operatorname{vol}_\omega}\braket{0,\varphi_t,\varphi_t,...,\varphi_t}_{(\operatorname{tr}R_0,\omega_0,...,\omega_0)}\\& \quad \quad- \frac{4\alpha_1\overline{C}_2}{(n+1)!}\braket{\varphi_t,...,\varphi_t}_{(\omega_0,...,\omega_0)}
	\\&=2\alpha_1\operatorname{M}^{\operatorname{Don}}_{\omega_t}(h_t,h_0)+\alpha_0\operatorname{M}^{\operatorname{cscK}}(\omega_t,\omega_0)\\& \quad \quad-4\pi\alpha_1\frac{1}{(n-1)!}\braket{\varphi_t,\dots,\varphi_t}_{(\operatorname{ch}_2(E,h_0))(\omega_0,\dots,\omega_0)}
	\\ & \quad +8\pi\alpha_1\frac{\mu(E)}{n!\operatorname{vol}_\omega}\braket{0,\varphi_t,\varphi_t,...,\varphi_t}_{(c_1(E,h_0),\omega_0,...,\omega_0)}\\& \quad \quad+ \frac{\alpha_1\overline{C}_1-4\alpha_1\overline{C}_2}{(n+1)!}\braket{\varphi_t,...,\varphi_t}_{(\omega_0,...,\omega_0)},
\end{align*}where $\overline{C}_1=\frac{4\pi\int_X\operatorname{ch}_2(E)\wedge\omega^{n-2}}{(n-2)!\operatorname{vol}_\omega},\overline{C}_2=2\pi\operatorname{rank}(E)\left(\frac{\mu(E)}{\operatorname{vol}_\omega}\right)^2$ are topological invariants.
\begin{proof}[\rm{\textbf{Proof of Theorem \ref{my}}}]
	By the discuss in section \ref{subsec:3}, such $(\mathcal{X},\mathcal{L})$ and $\Phi$ induces a subgeodesic $\varphi_t$ that is $\mathcal{C}^\infty$-compatible with $(\mathcal{X},c_1(\mathcal{L}))$. Thus the slope for the (modified) multivariate energy functional part is clear. The theorem then follows from \cite[Theorem A]{boucksom2019}.
\end{proof}
\begin{theorem}\label{mycor}
	Let $(X,\omega)$ be a compact K\"{a}hler manifold with $\omega$ located in the cohomology class $\Omega\in H^{1,1}(X,\mathbb{R})$, and $\pi:E\to X$ an irreducible holomoprhic vector bundle over $X$. Then for any Hermitian metric $h$ of $E$, any $(\mathcal{X}, \mathcal{A})$ a cohomological test configuration of $(X,\Omega)$ which is smooth and dominates $X\times \mathbb{P}^1$ with bimeromorphic morphism $\mu$, and any smooth subgeodesic ray $(\varphi_t)_{t\geq 0}$ compatible with $(\mathcal{X},\mathcal{A})$ (in the sense of Definition \ref{compatible}), we have\begin{align*}
		\lim_{t\to\infty}\frac{\mathcal{M}_I(b_t,b_0)}{t}&\leq\alpha_0\operatorname{DF}(\mathcal{X}, \mathcal{A})-4\pi\alpha_1\frac{1}{(n-1)!}(\mu^*\operatorname{pr}_1^*\operatorname{ch}_2(E)\cdot \mathcal{A}\cdot\dots\cdot \mathcal{A}) \\&\quad\quad+8\pi^2\alpha_1\frac{\mu(E)}{n!\operatorname{vol}_\Omega}(\mu^*\operatorname{pr}_1^*c_1(E)\cdot \mathcal{A}\cdot\dots\cdot \mathcal{A})
		)\\& \quad\quad+\frac{\alpha_1\overline{C}_1-4\alpha_1\overline{C}_2}{(n+1)!}(\mathcal{A}\cdot\dots\cdot \mathcal{A}),
	\end{align*}where $b_t=(\omega+dd^c\varphi_t,h)$, $\operatorname{DF}(\mathcal{X},\mathcal{A})$ is the Donaldson–Futaki invariant of the cohomological test configuration $(\mathcal{X},\mathcal{A})$.
\end{theorem}
\begin{proof}
	This is the corollary of (\ref{mytheorem1}) and \cite[Theorem 5.1]{sjostrom2018k}.
\end{proof}
\subsection{Proof of Theorem \ref{mytheorem2}}\label{subsec:8}
For the slope of the `Hermitian-Einstein' part $\operatorname{M}'$, we need to use another technique which is constructed in \cite{jonsson2022geodesicraysdonaldsonuhlenbeckyautheorem}. Suppose that $h_0$ is a Hermitian metric on $E$ and
\[
0 =: E_{m+1} \subset E_m \subset \dots \subset E_1 := E
\]
is a filtration of $E$ by holomorphic subbundles. Set $F_i=E_i/E_{i+1}$, we have a smooth orthogonal splitting $E=\bigoplus_{i=1}^mF_i$. Given any real numbers $\lambda_1>\cdots>\lambda_m$, we can use this smooth isomorphism to scale $h_0$ by $e^{t\lambda_i}$ on each $F_i$, which yields a smooth family $(h_t)_{t\geq 0}$ of Hermitian metrics. Under the orthogonal splitting, the Chern connection $h_t$ is given by (c.f. \cite[Lemma 3.2]{jonsson2022geodesicraysdonaldsonuhlenbeckyautheorem})\begin{align*}
	\begin{pmatrix}
		\beta_{11} & \beta_{12} & \beta_{13} & \cdots & \beta_{1m} \\
		e^{t(\lambda_2-\lambda_1)}\beta_{21} & \beta_{22} & \beta_{23} & \cdots & \beta_{2m} \\
		e^{t(\lambda_3-\lambda_1)}\beta_{31} & 	e^{t(\lambda_3-\lambda_2)}\beta_{32} & \beta_{33} & \cdots & \beta_{3m} \\
		\vdots & \vdots & \vdots & \vdots & \vdots \\
		e^{t(\lambda_m-\lambda_1)}\beta_{m1} & 	e^{t(\lambda_m-\lambda_2)}\beta_{m2} & 	e^{t(\lambda_m-\lambda_3)}\beta_{m3} & \cdots & \beta_{mm} 
	\end{pmatrix}
\end{align*}where $\beta_{ii}$ is the Chern connection form of the Hermitian metric induced by $(h_0)|_{F_i}$, and for $i > j$, $\beta_{ij}$ is a $\mathrm{Hom}(F_i, F_j)$-valued $(1,0)$-form whose $h_0$-adjoint gives $-\beta_{ji}$. Thus we can compute Donaldson's functional directly as follows (c.f. \cite[Theorem 3.3]{jonsson2022geodesicraysdonaldsonuhlenbeckyautheorem}):
\begin{align}\label{jonsson}
	\operatorname{M}^{\operatorname{Don}}_{\omega}(h_t,h_0) = 2\pi \sum_{i=1}^m \lambda_i \operatorname{rank}(F_i)(\mu_{F_i} - \mu_E)t - \sum_{1 \leqslant i < j \leqslant m} B_{ji}(1 - e^{-t(\lambda_i - \lambda_j)}),
\end{align}
where $B_{ji} = \int_X |\beta_{ji}|_{h_0}^2 \omega^n$ is a non-negative number.
\begin{proof}[\rm{\textbf{Proof of Theorem \ref{mytheorem2}}}]
	Note that the formula (\ref{jonsson}) holds for any K\"{a}hler form $\omega$. Thus along the ray $(\omega_t,h_t)$, where $\omega_t$ is any given smooth ray in $\mathcal{K}_{\check{I}}$, we have\begin{align*}
		\mathcal{M}^{\operatorname{HE}}(b_t,b_0)=2\alpha_1\operatorname{M}^{\operatorname{Don}}_{\omega_t}(h_t,h_0) &\ = 4\alpha_1\pi \sum_{i=1}^m \lambda_i \operatorname{rank}(F_i)(\mu_{F_i} - \mu_E)t \\&\ \  - 2\alpha_1\sum_{1 \leqslant i < j \leqslant m}\left(\int_X |\beta_{ji}|_{h_0}^2\omega^n_t\right)(1 - e^{-t(\lambda_i - \lambda_j)}).
	\end{align*}Since $\beta_{ji}$ is smooth and independent of $\omega_t$, the integral $\int_X |\beta_{ji}|_{h_0}^2\omega^n_t\leq Cn!\operatorname{vol}_\Omega$ i.e. uniformly bounded. The theorem follows when $t\to\infty$.
\end{proof}
\begin{remark}
	The asymptotic slope of $\operatorname{M}'$ may also be evaluated using the results of \cite{hashimoto2022quot}, provided that $X$ is polarized and $E(k) := E \otimes \mathcal{O}_X(k)$ is globally generated, with the smooth $\operatorname{Herm}^+(E)$-ray given by a one-parameter subgroup $\sigma: \mathbb{C}^* \to \operatorname{SL}(H^0(X,E(k))^*)$. Furthermore, if we relax the smooth condition of $\operatorname{Herm}^+(E)$-rays, we can get the subsheaf version of Theorem \ref{mytheorem2} by \cite[Theorem A]{jonsson2022geodesicraysdonaldsonuhlenbeckyautheorem}.
\end{remark}

\bibliographystyle{spmpsci.bst}
\bibliography{references}

\providecommand{\bysame}{\leavevmode\hbox to3em{\hrulefill}\thinspace}
\providecommand{\MR}{\relax\ifhmode\unskip\space\fi MR }
\providecommand{\MRhref}[2]{%
  \href{http://www.ams.org/mathscinet-getitem?mr=#1}{#2}
}
\providecommand{\href}[2]{#2}
\begin{thebibliography}{10}

\bibitem{alvarez2013coupled}
L.~Alvarez-Consul, M.~Garcia-Fernandez, and O.~Garcia-Prada, \emph{Coupled
  equations for {K}\"{a}hler metrics and {Y}ang--{M}ills connections}, Geom.
  Topol. \textbf{17} (2013), 2731--2812.

\bibitem{BedfordTaylor76}
E.~Bedford and B.~A. Taylor, \emph{The {D}irichlet problem for a complex
  {M}onge-{A}mp{\`e}re equation}, Invent. Math. \textbf{37} (1976), 1--44.

\bibitem{BedfordTaylor82}
\bysame, \emph{A new capacity for plurisubharmonic functions}, Acta Math.
  \textbf{149} (1982), 1--40.

\bibitem{boucksom2017uniform}
S.~Boucksom, T.~Hisamoto, and M.~Jonsson, \emph{Uniform {K}-stability,
  {D}uistermaat--{H}eckman measures and singularities of pairs}, Ann. Inst.
  Fourier (Grenoble) \textbf{67} (2017), 743--841.

\bibitem{boucksom2019}
\bysame, \emph{Uniform {K}-stability and asymptotics of energy functionals in
  {K}\"ahler geometry}, J. Eur. Math. Soc. (JEMS) \textbf{21} (2019),
  2905--2944.

\bibitem{chen2000lower}
X.X. Chen, \emph{On the lower bound of the {M}abuchi energy and its
  application}, Int. Math. Res. Not. \textbf{12} (2000), 607--623.

\bibitem{chen2015kahler}
X.X. Chen, S.~Donaldson, and S.~Sun, \emph{K\"{a}hler-{E}instein metrics on
  {F}ano manifolds. {I}: {A}pproximation of metrics with cone singularities},
  J. Amer. Math. Soc. \textbf{28} (2015), 183--197.

\bibitem{chen2015kahler2}
\bysame, \emph{K{\"a}hler-{E}instein metrics on {F}ano manifolds. {II}:
  {L}imits with cone angle less than 2$\pi$}, J. Amer. Math. Soc. \textbf{28}
  (2015), 199--234.

\bibitem{chen2015kahler3}
\bysame, \emph{K{\"a}hler-{E}instein metrics on {F}ano manifolds. {III}:
  {L}imits as cone angle approaches 2$\pi$ and completion of the main proof},
  J. Amer. Math. Soc. \textbf{28} (2015), 235--278.

\bibitem{Dervan2023}
R.~Dervan, \emph{Stability conditions for polarised varieties}, Forum Math.
  Sigma \textbf{11} (2023), Paper No. e104, 57.

\bibitem{donaldson1985anti}
S.~K. Donaldson, \emph{Anti self-dual {Y}ang-{M}ills connections over complex
  algebraic surfaces and stable vector bundles}, Proc. Lond. Math. Soc.
  \textbf{3} (1985), 1--26.

\bibitem{donaldson1999symmetric}
\bysame, \emph{Symmetric spaces, {K}\"ahler geometry and {H}amiltonian
  dynamics}, Northern {C}alifornia {S}ymplectic {G}eometry {S}eminar, Amer.
  Math. Soc. Transl. Ser. 2, vol. 196, Amer. Math. Soc., Providence, RI, 1999,
  pp.~13--33.

\bibitem{donaldson2002scalar}
\bysame, \emph{Scalar curvature and stability of toric varieties}, J.
  Differential Geom. \textbf{62} (2002), 289--349.

\bibitem{garcia2011coupled}
M.~Garcia-Fernandez, \emph{Coupled equations for {K}\"{a}hler metrics and
  {Y}ang-{M}ills connections ({T}hesis)}, arXiv preprint arXiv:1102.0985
  (2011).

\bibitem{hashimoto2022quot}
Y.~Hashimoto and J.~Keller, \emph{Quot-scheme limit of {F}ubini-{S}tudy metrics
  and {D}onaldson's functional for vector bundles}, \'Epijournal G\'eom.
  Alg\'ebrique \textbf{5} (2021), Art. 21, 38.

\bibitem{jonsson2022geodesicraysdonaldsonuhlenbeckyautheorem}
M.~Jonsson, N.~McCleerey, and S.~Shivaprasad, \emph{Geodesic {R}ays in the
  {D}onaldson--{U}hlenbeck--{Y}au theorem}, arXiv preprint arXiv:2210.09246
  (2022).

\bibitem{kobayashi2014differential}
S.~Kobayashi, \emph{Differential geometry of complex vector bundles}, Princeton
  University Press, 2014.

\bibitem{kobayashinomizu1969foundations}
S.~Kobayashi and K.~Nomizu, \emph{Foundations of differential geometry},
  Interscience, 1969.

\bibitem{Lubke1983}
M.~L\"ubke, \emph{Stability of {E}instein-{H}ermitian vector bundles},
  Manuscripta Math. \textbf{42} (1983), 245--257.

\bibitem{mabuchi1986k}
T.~Mabuchi, \emph{K-energy maps integrating {F}utaki invariants}, Tohoku Math.
  J. (2) \textbf{38} (1986), 575--593.

\bibitem{singer1959geometric}
I.~M. Singer, \emph{The geometric interpretation of a special connection},
  Pacific J. Math. \textbf{9} (1959), 585--590.

\bibitem{dyrefelt2017k}
Z.~Sj{\"o}str{\"o}m~Dyrefelt, \emph{K-stability and {K}{\"a}hler manifolds with
  transcendental cohomology class}, Ph.D. thesis, Universit{\'e} Paul
  Sabatier-Toulouse III, 2017.

\bibitem{sjostrom2018k}
\bysame, \emph{K-semistability of csck manifolds with transcendental cohomology
  class}, J. Geom. Anal. \textbf{28} (2018), 2927--2960.

\bibitem{tian1997kahler}
G.~Tian, \emph{K{\"a}hler-{E}instein metrics with positive scalar curvature},
  Invent. Math. \textbf{130} (1997), 1--37.

\bibitem{uhlenbeck1986existence}
K.~Uhlenbeck and S.~T. Yau, \emph{On the existence of
  {H}ermitian-{Y}ang-{M}ills connections in stable vector bundles}, Comm. Pure
  Appl. Math. \textbf{39} (1986), S257--S293.

\bibitem{yau}
S.~T. Yau, \emph{Nonlinear analysis in geometry}, Enseign. Math. (2)
  \textbf{33} (1987), 109--158.

\end{thebibliography}
\end{document}